\documentclass[a4paper,12pt,reqno]{amsart}

\usepackage{color}
\usepackage{amsmath}
\usepackage{amssymb}
\usepackage{amsfonts}
\usepackage{graphicx}
\usepackage{mathtools}
\usepackage[colorlinks]{hyperref}
\renewcommand\eqref[1]{(\ref{#1})} 
\usepackage{booktabs}
\usepackage{lscape}

\usepackage{mathrsfs}
\usepackage{comment}

\graphicspath{ {images/} }
\newcommand{\R}{\mathbb{R}}

\title[Anomaly of the fractional heat propagator]{Anomaly of the fractional heat propagator in abstract settings}

\author[J.E. Restrepo]{Joel E. Restrepo}
\address{\textcolor[rgb]{0.00,0.00,0.84}{Joel E. Restrepo \newline Department of Mathematics \newline CINVESTAV, Mexico City, Mexico}}
\email{\textcolor[rgb]{0.00,0.00,0.84}{joel.restrepo@cinvestav.mx}}


\newtheoremstyle{theorem}
{10pt}          
{10pt}  
{\sl}  
{\parindent}     
{\bf}  
{. }    
{ }    
{}     
\theoremstyle{theorem}
\newtheorem{theorem}{Theorem}

\numberwithin{equation}{section}
\theoremstyle{plain}
\newtheorem{thm}{Theorem}[section]

\theoremstyle{definition}
\newtheorem{defn}[thm]{Definition}

\newtheoremstyle{defi}
{10pt}          
{10pt}  
{\rm}  
{\parindent}     
{\bf}  
{. }    
{ }    
{}     
\theoremstyle{defi}

\newtheorem{remark}[theorem]{Remark}

\begin{document}
 	\begin{abstract}
We study the following time-fractional heat equation: 
\begin{equation*}
^{C}\partial_{t}^{\alpha}u(t)+\mathscr{L}u(t)=0,\quad u(0)=u_0\in X, \quad t\in[0,T],\quad T>0,\quad 0<\alpha<1, 
\end{equation*}
where $^{C}\partial_{t}^{\alpha}$ is the Djrbashian-Caputo fractional derivative, $X$ is a complex Banach space and $\mathscr{L}:\mathcal{D}(\mathscr{L})\subset X\to X$ is a closed linear operator. The solution operator of the equation above is given by the strongly continuous operator $E_\alpha(-t^{\alpha}\mathscr{L})$ for any $t\geqslant0$, closely related with the Mittag-Leffler function $E_\alpha(-x)$ for $x\geqslant0.$ There are different ways to present explicitly this operator and one of the most popular is given in terms of the $C_0$-semigroup generated by $-\mathscr{L}$ $\big(\{e^{-t\mathscr{L}}\}_{t\geqslant0}\big)$ as follows: 
\[
E_\alpha(-t^{\alpha}\mathscr{L})=\int_0^{+\infty}M_{\alpha}(s)e^{-st^{\alpha}\mathscr{L}}{\rm d}s,\quad t\geqslant0,    
\]
where $M_{\alpha}(s)$ is a Wright-type function. We will see that the latter expression is not always optimal (regarding restrictions: endpoint lost) to estimate different norms. An additional restriction appears while bounding the above integral, which can be avoided by using directly the function itself and its well-known uniform bound $|E_{\alpha}(-x)|\leqslant \frac{C}{1+x},$ $x\geqslant0.$       
\end{abstract}
	\maketitle
	\tableofcontents

\section{Introduction}

In the study of fractional differential equations in Banach spaces \cite{thesis2001,thesis,pruss}, it may be noted that many authors in the literature (see \cite{navier,chen,[62],section3,[79]} and the references therein) have strategically rewritten fractional heat propagators in terms of classical heat ones, since estimates for the latter are already well-known in many cases. This issue can also be seen in the study of some subordination principles. We proceed to describe more clearly this idea as follows. First of all, we consider the following time-fractional heat equation:
\begin{equation}\label{proto}
^{C}\partial_{t}^{\alpha}u(t)+\mathscr{L}u(t)=0,\quad u(0)=u_0\in X,\quad t\in[0,T],\quad T>0,\quad 0<\alpha<1, 
\end{equation}
where the Djrbashian-Caputo fractional derivative is defined as 
\[^{C}\partial_{t}^{\alpha}u(t)=\int_0^t \frac{(t-s)^{-\alpha}}{\Gamma(1-\alpha)}\partial_{s}u(s){\rm d}s,\] 
$X$ is a Banach space over $\mathbb{C}$ and $\mathscr{L}:\mathcal{D}(\mathscr{L})\subset X\to X$ is a closed  linear operator which is densely defined. The assumption on the density can be dropped to develop the idea of this paper. In this case, it is just necessary to follow the ideas from \cite{JEE2002,section3} for the so-called almost sectorial operators.  

The solution operator (propagator) of equation \eqref{proto} is associated with the Mittag-Leffler function:
\begin{equation}\label{mittag}
E_{\alpha}(z)=\sum_{k=0}^{+\infty} \frac{z^k}{\Gamma(\alpha k+1)},\quad z\in\mathbb{C},\quad \Re(\alpha)>0,
\end{equation}
which is absolutely and locally uniformly convergent for the given parameters \cite{mittag}. Note that the integral representation of the Mittag-Leffler function is given by   
\begin{equation}\label{integral-mittag}
E_{\alpha}(z)=\frac{1}{2\pi i}\int_{\mathfrak{H}}e^{\gamma}\gamma^{\alpha-1}(\gamma^{\alpha}-z)^{-1}{\rm d}\gamma,
\end{equation}
where $\mathfrak{H}$ is a suitable Hankel path \cite[Formula (3.4.12)]{mittag}.

From now on we focus on working with positive sectorial operators \cite{functionalcalculus}, i.e. those operators for which there exist $N\geqslant1$ and $\phi\in(0,\pi/2)$ such that 
\[
S_{\phi}:=\{\lambda\in\mathbb{C}:\phi\leqslant|\text{arg}(\lambda)|\leqslant\pi\}\subset\rho(\mathscr{L})\footnote{The resolvent set of $\mathscr{L}$.}
\]
and 
\[
\|(\lambda-\mathscr{L})^{-1}\|_{\mathcal{L}(X)}\leqslant \frac{N}{|\lambda|},\quad \lambda\in S_{\phi}\setminus\{0\},
\]
where $\mathcal{L}(X)$ denotes the space of all bounded linear operators from $X$ to $X$, and its norm is given by $\|\cdot\|_{\mathcal{L}(X)}.$
Thus, for the above type of operators, it was proved in \cite[Theorem 2.41]{thesis} (see also  \cite[Lemma 2.1]{tmn}) that the global mild solution of \eqref{proto} is given by the well defined strongly continuous operator 
\begin{equation}\label{heatf}
E_\alpha(-t^{\alpha}\mathscr{L}):=\frac{1}{2\pi i}\int_{H}e^{st}s^{\alpha-1}(s^{\alpha}+\mathscr{L})^{-1}{\rm d}s,\quad t\geqslant0,    
\end{equation}
where $H$ is the Hankel path of \cite[Formula (2.5)]{thesis} and it is contained in the resolvent set $\rho(-\mathscr{L})$. Moreover 
\[
\sup_{t\geqslant0}\|E_\alpha(-t^{\alpha}\mathscr{L})\|_{\mathcal{L}(X)}\leqslant M,\quad\text{for some constant $M>0$ (uniform on $\alpha$).}
\]
Representation \eqref{heatf} and its properties are initially inspired by the early works of Bajlekova, see e.g. \cite[Chapters 2 and 3]{thesis2001}. These ideas came from a more general theory of evolutionary integral equations \cite[Chapter 2]{pruss}. We highlight that the propagator \eqref{heatf} can be also defined by using the functional calculus \cite{BorelFunctional}.

We point out that the notation in \eqref{heatf} is consistent (convenient) since it coincides with the classical integral representation of the Mittag-Leffler function if by abuse of notation we replace the operator $\mathscr{L}$ by a real number (see formula \eqref{integral-mittag}). The propagator in \eqref{heatf} can also be rewritten as follows \cite[Theorem 3.1]{thesis2001} (see \cite[Theorem 2.42]{thesis}, \cite[Lemma 9]{[62]} or \cite{[79]})
\begin{equation}\label{heatfa}
E_\alpha(-t^{\alpha}\mathscr{L})=\int_0^{+\infty}M_{\alpha}(s)e^{-st^{\alpha}\mathscr{L}}{\rm d}s,\quad t\geqslant0,    
\end{equation}
where $\{e^{-t\mathscr{L}}\}_{t\geqslant0}$ is the $C_0-$ semigroup generated by $-\mathscr{L},$ and
\begin{equation}\label{wright}
M_{\alpha}(z)=\sum_{n=0}^{+\infty}\frac{(-z)^n}{n!\Gamma(-\alpha n+1-\alpha)},\quad z\in\mathbb{C},\quad 0\leqslant\alpha<1,
\end{equation}
is the Wright-type function which is convergent in the whole $z$-complex plane \cite{1940}. Some of the basic properties of this function are:
\[
M_\alpha(x)\geqslant 0\quad\text{for all}\quad x\in(0,+\infty),\quad \int_0^{+\infty}M_\alpha(s){\rm d}s=1,
\]
and 
\begin{equation}\label{estimate}
\int_0^{+\infty}s^{\gamma}M_{\alpha}(s){\rm d}s=\frac{\Gamma(\gamma+1)}{\Gamma(\gamma\alpha+1)},\quad \gamma>-1,\quad 0\leqslant \alpha<1.
\end{equation}

It is important to emphasise that, in the limit case $\alpha\to1^{-}$, the function $M_{\alpha}(x)$, for $x\in\mathbb{R}^+$, tends to the Dirac generalised function $\delta(x-1).$ The function in \eqref{wright} has also been called the Mainardi function \cite{1993}, however Mainardi recognised later that he was not aware of Wright's 1940 work \cite{1940} and thus he thought that it was a new function although it was not. Nevertheless, it is necessary to mention that Mainardi demonstrated through many years the great value of this function in studying fractional differential equations and its applications, as detailed in his monograph \cite{Mainardi}. It is worth noting that the following unidimensional expression of \eqref{heatfa}:
\[
E_\alpha(-x)=\int_0^{+\infty}M_{\alpha}(s)e^{-sx}{\rm d}s,\quad x\geqslant0,
\] 
is consistent with the classical result given by Pollard in 1948 \cite{Pollard}, who also showed that the Mittag-Leffler function $E_\alpha(-x)$ is completely monotonic.

At this stage, many researchers prefer to use the expression \eqref{heatfa} in their works since it depends on the $C_0-$ semigroup. Therefore, estimates for the fractional heat propagator will follow from the estimates of the classical heat propagator which in several cases are well-known and even sharp. Moreover, it is easy to handle and estimates reduce to take care of the convergence of a real value integral. Nevertheless, we will show that expression \eqref{heatfa} is not always optimal to estimate some norms since stronger restrictions are needed (endpoint lost). In the next section, we focus on illustrating this phenomenon in the setting of locally compact groups. In fact, we provide the main result of this paper on fractional heat equations. We estimate the time-decay rate for the considered equations by using two different representations of the fractional heat propagator. Finally, we contrast both results and show their advantages and disadvantages.   

\section{Fractional heat equations and their time-decay estimates}\label{mainresult}

We are now ready to discuss the mentioned problem (see the introduction) in the context of locally compact groups. In this scenario, we will show explicitly the disadvantage in some estimates of using the representation \eqref{heatfa} for the fractional heat propagator. Thus, we first recall some results on the classical heat equation and its time-decay rate. After that, we continue by its fractional counterpart. In the end, we compare all the results.  

\subsection{Heat equation and its time-decay rate} 

As an application of some spectral multipliers results on a locally compact separable unimodular group $G$, it was shown recently in \cite{RR2020} that the  $L^p-L^q$ $(1<p\leqslant 2\leqslant q<+\infty)$ norm estimates for the solution operator of the following $\mathscr{L}$-heat equation 
\begin{align}\label{asterisco}
\begin{split}
\partial_{t}w(t,x)+\mathscr{L}w(t,x)&=0, \quad t>0,\,\, x\in G, \\
w(t,x)|_{_{_{t=0}}}&=w_0(x),
\end{split}
\end{align}
can be reduced to the time asymptotics of its propagator \cite[Theorems 5.1 and 6.1]{RR2020} in the noncommutative Lorentz space norm \cite{[51]}, which involves calculating the trace of the spectral projections of the operator $\mathscr{L}$. Indeed, the following estimate (time-decay rate) was obtained for the solution $w(t,x)=e^{-t\mathscr{L}}w_0(x)$ of equation \eqref{asterisco}:
\begin{equation}\label{asym-heat}
\|e^{-t\mathscr{L}}w_0\|_{L^q(G)}\leqslant C_{\lambda,p,q}t^{-\lambda\left(\frac{1}{p}-\frac{1}{q}\right)}\|w_0\|_{L^p(G)},\quad t>0,
\end{equation}
whenever
\begin{equation}\label{trace}
\tau\big(E_{(0,s)}(\mathscr{L})\big)\lesssim s^{\lambda},\quad s\to+\infty,\quad\text{for some}\quad \lambda>0.
\end{equation}
The symbol $\tau$ denotes the trace over the spectral projections $E_{(0,s)}(\mathscr{L})$ of the operator $\mathscr{L}$ in the interval $(0,s).$ This trace appears frequently in the construction of the Fourier multipliers on locally compact groups \cite{RR2020}. In fact, it is considered $M\subset\mathfrak{L}(\mathcal{H})$ to be a semifinite von Neumann algebra acting over the Hilbert space $\mathcal{H}$ with a trace $\tau,$ and $\mathfrak{L}(\mathcal{H})$ stands for the set of linear operators defined on $\mathcal{H}$. More details about this trace and its link with the von Neumann algebras can be found in \cite{von,[46],[47]}.

Let us recall that $L^p-L^q$ boundedness of the solution of \eqref{asterisco} is conditioned by the assumption:
\[
\sup_{t>0}\sup_{s>0}\big[\tau\big(E_{(0,s)}(\mathscr{L})\big)\big]^{\frac{1}{p}-\frac{1}{q}}e^{-ts}<+\infty.
\]
The considered operator (unbounded) can be any positive (nonnegative and self-adjoint) linear left invariant operator acting on $G$ with the possibility (generality) of having either continuous or discrete spectrum. We provide some examples (with different groups) where the trace of an operator has such a behaviour as in \eqref{trace}. 

\begin{itemize}
\item[1.] In the Euclidean space $\R^n$, we consider the Laplacian $\Delta_{\R^n}$ on $\R^n$.  It was proved in \cite[Example 7.3]{RR2020} that
\[
\tau\big(E_{(0,s)}(\Delta_{\R^n})\big)\lesssim s^{n/2},\quad s\to+\infty.
\]
    \item[2.] The sub-Laplacian $\Delta_{sub}$ on a compact Lie group. In \cite{[35]}, it can be found that
\[
\tau\big(E_{(0,s)}(-\Delta_{sub})\big)\lesssim s^{Q/2},\quad s\to+\infty,
\]
where $Q$ is the Hausdorff dimension of $G.$ 

\item[3.] The positive sub-Laplacian $\mathscr{L}$ on the Heisenberg group $\mathbb{H}^n$. By \cite[Formula (7.17)]{RR2020}, we know that
\[
\tau\big(E_{(0,s)}(\mathscr{L})\big)\lesssim s^{n+1},\quad s\to+\infty.
\]

\item[4.] A positive Rockland operator $\mathcal{R}$ of order $\nu$ on a graded Lie group. By \cite[Theorem 8.2]{david}, we get that
\[
\tau\big(E_{(0,s)}(\mathcal{R})\big)\lesssim s^{Q/\nu},\quad s\to+\infty,
\]
where $Q$ is the homogeneous dimension of $G$. 
\end{itemize}

Some other more technical examples can be found e.g. in \cite[Examples  2.2 and 3.2]{Marianna} or \cite[Proposition 0.3]{david2}. So, we have shown several operators over different groups satisfying condition \eqref{trace}. Here the type of groups that we can consider is very large. For instance, compact, semi simple, exponential, nilpotent, some solvable ones, real algebraic, and many more.

\medskip\noindent Below we treat the time-fractional heat equation and recall how its solution operator can be associated with the classical heat propagator. 

\subsection{Time-fractional heat equation and its time-decay rate} In the recent papers \cite{SRR-1,SRR-2}, it was studied the following time-fractional $\mathscr{L}$-heat equation:
\begin{align}\label{heat-f}
\begin{split}
^{C}\partial_{t}^{\alpha}w(t,x)+\mathscr{L}w(t,x)&=0,\quad t>0,\,\, x\in G,  \\
w(t,x)|_{_{_{t=0}}}&=w_0(x),\quad w_0\in L^p(G), 
\end{split}
\end{align}
where $0<\alpha\leqslant1,$ $1<p\leqslant 2$ and $\mathscr{L}$ is any positive linear left invariant operator on $G$ (a locally compact separable unimodular group). Notice that the integro-differential operator in the left hand side of equation \eqref{heat-f} can be also rewritten by using the Riemann-Liouville integral \cite{samko} as $\prescript{RL}{0}I^{1-\alpha}\partial_t w(t,x)$. This implies that for $\alpha=1$, we get $\partial_t w(t,x)$ since the operator $\prescript{RL}{0}I^{0}=\textrm{Id}$ acts as the identity. We exclude the value $\alpha=1$ from consideration since then problem \eqref{heat-f} is the same as \eqref{asterisco} and the result in this case is known \cite{RR2020}. The $L^p-L^q$ $(1<p\leqslant 2\leqslant q<+\infty)$ norm estimates for the solution of problem \eqref{heat-f} were obtained in \cite[Theorem 5]{SRR-1} (see also \cite[Theorem 3.3]{SRR-2}). The result reads as follows:
\begin{thm}\label{heat-f-thm}
Let $G$ be a locally compact separable unimodular group and let $1<p\leqslant 2\leqslant q<+\infty$. Let $\mathscr{L}$ be any positive left invariant operator on $G$ (unbounded) such that 
\begin{equation}\label{cond-1}
\sup_{t>0}\sup_{s>0}\big[\tau\big(E_{(0,s)}(\mathscr{L})\big)\big]^{\frac{1}{p}-\frac{1}{q}}E_\alpha(-t^{\alpha}s)<+\infty.
\end{equation}
If $0<\alpha<1$ and $w_0\in L^p(G)$ then the solution for the time-fractional $\mathscr{L}$-heat equation \eqref{heat-f} is given by 
\[
w(t,x)=E_\alpha(-t^{\alpha}\mathscr{L})w_0(x),\quad t>0,\,\,x\in G,
\]
and also it is in $L^q(G).$ In particular, if the condition \eqref{trace} holds, we obtain the following time decay rate for the solutions of \eqref{heat-f}:  
\[
\|w(t,\cdot)\|_{L^q(G)}\leqslant C_{\alpha,\lambda,p,q}t^{-\alpha\lambda\left(\frac{1}{p}-\frac{1}{q}\right)}\|w_0\|_{L^p(G)},\quad\text{whenever}\quad \frac{1}{\lambda}\geqslant\frac{1}{p}-\frac{1}{q}.
\]
\end{thm}

Let us recall that the solution operator $\{E_\alpha(-t^{\alpha}\mathscr{L})\}_{t\geqslant0}\subset \mathcal{L}(L^p(G))$ satisfies the following properties:
\begin{enumerate}
    \item $E_\alpha(-t^{\alpha}\mathscr{L})$ is strongly continuous for $t\geqslant0$ and $E_\alpha(0)=I;$
    \item $E_\alpha(-t^{\alpha}\mathscr{L})\mathcal{D}(\mathscr{L})\subset\mathcal{D}(\mathscr{L})$ and $\mathscr{L}E_\alpha(-t^{\alpha}\mathscr{L})w=E_\alpha(-t^{\alpha}\mathscr{L})\mathscr{L}w$ for any $w\in\mathcal{D}(\mathscr{L})$, $t\geqslant0;$
    \item $E_\alpha(-t^{\alpha}\mathscr{L})w$ is a solution of equation \eqref{heat-f} for any $w\in\mathcal{D}(\mathscr{L})$, $t\geqslant0.$
\end{enumerate}

\begin{remark}
The restriction $\frac{1}{\lambda}\geqslant\frac{1}{p}-\frac{1}{q}$ found in Theorem \ref{heat-f-thm} will play an important role in the development of this short note. For the time decay rate obtained in the latter theorem was used the inequality \cite[Theorem 1.6]{page 35}: 
\begin{equation}\label{uniform-estimate}
|E_{\alpha}(-x)|\leqslant \frac{C}{1+x},\quad x\geqslant0,\quad \alpha<2,\quad\text{where $C$ is a positive constant.}
\end{equation}
\end{remark}

\subsection{Alternative approach to study the time-fractional heat equation} It was discussed in the introduction that the solution operator of the time-fractional heat equation \eqref{heat-f} in a complex Banach space can be given by the relation \eqref{heatfa} which involves the heat propagator. Hence, if we consider the space $L^p(G)$ $(1<p\leqslant2)$ where $G$ is a locally compact group, we can then establish the following result.

\begin{thm}\label{thm-f-heat}
Let $G$ be a locally compact separable unimodular group and let $1<p\leqslant 2\leqslant q<+\infty$. Let $\mathscr{L}$ be any positive left invariant operator on $G$ (unbounded) such that 
\begin{equation}\label{cond-2}
\sup_{t>0}\sup_{s>0}\big[\tau\big(E_{(0,s)}(\mathscr{L})\big)\big]^{\frac{1}{p}-\frac{1}{q}}e^{-ts}<+\infty.
\end{equation}
If $0<\alpha<1$ and $w_0\in L^p(G)$ then the solution for the time-fractional $\mathscr{L}$-heat equation \eqref{heat-f} is given by 
\[
w(t,x)=E_\alpha(-t^{\alpha}\mathscr{L})w_0(x)=\int_0^{+\infty}M_{\alpha}(s)e^{-st^{\alpha}\mathscr{L}}w_0(x)\,{\rm d}s,\quad t>0,\,\,x\in G,
\]
and also it is in $L^q(G).$ In particular, if the condition \eqref{trace} holds, we get the following time decay rate for the solutions of \eqref{heat-f}:  
\[
\|w(t,\cdot)\|_{L^q(G)}\leqslant C_{\alpha,\lambda,p,q}t^{-\alpha\lambda\left(\frac{1}{p}-\frac{1}{q}\right)}\|w_0\|_{L^p(G)},\quad\text{whenever}\quad \frac{1}{\lambda}>\frac{1}{p}-\frac{1}{q}.
\]
\end{thm} 
\begin{proof}
It is clear that the solution operator of equation \eqref{heat-f} is a bounded operator from $L^p(G)$ to $L^{p}(G),$ and it is given by the representation \eqref{heatfa}. Let us now show that the solution is in $L^q(G)$ whenever the data is in $L^p(G).$ Indeed, suppose that $w_0\in L^p(G).$ Thus, by using the estimate \eqref{asym-heat}, it yields that
\begin{align}
\|E_\alpha(-t^{\alpha}\mathscr{L})w_0\|_{L^q(G)}&\leqslant \int_0^{+\infty}M_{\alpha}(s)\|e^{-st^{\alpha}\mathscr{L}}w_0\|_{L^q(G)}{\rm d}s,\quad t>0, \nonumber\\    
&\leqslant C_{\alpha,\lambda,p,q}\|w_0\|_{L^p(G)}t^{-\alpha\lambda\left(\frac{1}{p}-\frac{1}{q}\right)}\int_0^{+\infty}M_{\alpha}(s)s^{-\lambda\left(\frac{1}{p}-\frac{1}{q}\right)}{\rm d}s, \label{333}
\end{align}
and the last integral is convergent (see estimate \eqref{estimate}) when 
$-\lambda\left(\frac{1}{p}-\frac{1}{q}\right)>-1,$ i.e. $\frac{1}{\lambda}>\frac{1}{p}-\frac{1}{q},$ which finishes the proof. 
\end{proof}

\subsection{Conclusion}

We contrast the results stated in Theorems \ref{heat-f-thm} and \ref{thm-f-heat} on the solution of the time-fractional heat equation \eqref{heat-f}.

Before giving the final remark of this paper, we compare the conditions \eqref{cond-1} and \eqref{cond-2}. First, note that, in general, we can not compare the functions $e^{-t}$ and $E_\alpha(-t)$ for $t\geqslant0,$ see e.g. \cite{555}. Usually, they exceed each other in two different frames. Having in mind the behaviour of the condition \eqref{trace}, we obtain that 
\[
\sup_{s>0}\big[\tau\big(E_{(0,s)}(\mathscr{L})\big)\big]^{\frac{1}{p}-\frac{1}{q}}E_\alpha(-t^{\alpha}s)\leqslant C_{\lambda,p,q} \sup_{s>0}\frac{s^{\lambda\left(\frac{1}{p}-\frac{1}{q}\right)}}{1+t^\alpha s},
\]
due to the uniform estimate \eqref{uniform-estimate}. For ${\frac{1}{p}-\frac{1}{q}}=1/\lambda$, the supremum is bounded by $t^{-\alpha}.$ Also, the above supremum is attained at $s=\frac{\lambda t^{-\alpha}}{{\frac{1}{p}-\frac{1}{q}}-\lambda}$ whenever $\frac{1}{\lambda}>\frac{1}{p}-\frac{1}{q}$, hence
\[
\sup_{s>0}\big[\tau\big(E_{(0,s)}(\mathscr{L})\big)\big]^{\frac{1}{p}-\frac{1}{q}}E_\alpha(-t^{\alpha}s)\leqslant C_{\alpha,\lambda,p,q}t^{-\alpha\lambda\left(\frac{1}{p}-\frac{1}{q}\right)}. 
\]
Therefore, we can see clearly how the constrain appears in Theorem \ref{heat-f-thm}. Also, we can check easily that 
\[
\sup_{s>0}\big[\tau\big(E_{(0,s)}(\mathscr{L})\big)\big]^{\frac{1}{p}-\frac{1}{q}}e^{-ts}\lesssim \sup_{s>0}s^{\lambda\left(\frac{1}{p}-\frac{1}{q}\right)}e^{-ts}=\left(t^{-1}\lambda\left(\frac{1}{p}-\frac{1}{q}\right)\right)^{\lambda\left(\frac{1}{p}-\frac{1}{q}\right)}e^{-\lambda\left(\frac{1}{p}-\frac{1}{q}\right)},
\]
which shows that does not appear any restriction. But then, this generality is lost when it is used representation \eqref{heatfa} of the fractional heat propagator in Theorem \ref{thm-f-heat}. 
 
Hence, in Theorem \ref{thm-f-heat}, it was used the alternative representation \eqref{heatfa} of $E_\alpha(-t^{\alpha} \mathscr{L})$ for the solution of \eqref{heat-f}. In this case, we are losing the possibility of having the equality in the restriction $\frac{1}{\lambda}>\frac{1}{p}-\frac{1}{q}$ which was stated in Theorem \ref{heat-f-thm}. This is mainly happening for the 
constraint found in the convergence of the integral involving the Wright-type function in the proof of Theorem \ref{thm-f-heat}. Note that this latter fact is not affecting the condition $\frac{1}{\lambda}\geqslant\frac{1}{p}-\frac{1}{q}$ in Theorem \ref{heat-f-thm} since there it was possible to get a better constrain. Indeed,  in the proof (see \cite[Theorem 3.3]{SRR-2}) of the boundedness it was possible to separate the contribution of the operator and the function itself by using some contemporary Fourier multiplier results from \cite{RR2020, SRR-2}.  

In one of the most simple cases, we can take $G=\mathbb{R}^n$ and the operator being the Laplacian $\Delta_{\R^n}.$ By Theorem \ref{heat-f-thm}, we have that the decay estimate of problem \eqref{heat-f} coincides with the sharp estimate (time-decay) given in \cite[Theorem 3.3, items (i) and (ii)]{uno2} whenever $\frac{2}{n}\geqslant\frac{1}{p}-\frac{1}{q}$. Now if we use representation \eqref{heatfa} as we have seen in \eqref{333} (i.e. Theorem \ref{thm-f-heat}), we then lose the equality case (endpoint restriction). 

We can then conclude that boundedness results for the propagators of the time-fractional heat equations are far to be optimal regarding restrictions if it is used the explicit representation \eqref{heatfa} of its propagator. This open a more general question on which representations and methods to use for these propagators in abstract settings for getting the most optimal results e.g. in the $L^p-L^p$ or $L^p-L^q$ boundedness of the solution operator of \eqref{heat-f}. Of course, the latter question can be extended to the study of some other properties as well.    


\end{document}